\documentclass[11pt,reqno]{amsart}
\usepackage{mathptmx}      
\usepackage{amsmath}
\usepackage{amssymb}
\usepackage{mathrsfs}
\usepackage{url}
\usepackage{eucal}

\def\Rc{\mathrel{\mathscr{R}}}
\def\Lc{\mathrel{\mathscr{L}}}
\def\Jc{\mathrel{\mathscr{J}}}

\newcommand{\sgp}{semi\-group}
\newcommand{\sgps}{semi\-groups}
\newcommand{\is}{involutory semi\-group}
\newcommand{\iss}{involutory semi\-groups}
\newcommand{\fis}{finite involutory semi\-group}

\newcommand{\fss}{finite \sgps}
\newcommand{\fus}{finite unary \sgp}
\newcommand{\fuss}{finite unary \sgps}
\newcommand{\fb}{finitely based}
\newcommand{\fbp}{finite basis problem}
\newcommand{\fg}{finitely generated}
\newcommand{\nfb}{non\-finitely based}
\newcommand{\infb}{inherently non\-finitely based}
\newcommand{\TSL}{\ensuremath{\mathcal{T\kern-.5pt S\kern-.5pt L}}}
\newcommand{\TB}{\ensuremath{\mathcal{T\kern-.5pt B}_2^1}}
\newcommand{\FI}{\ensuremath{\mathcal{F\kern-.3pt I}}}

\newtheorem{Thm}{Theorem}[section]
\newtheorem{Prop}[Thm]{Proposition}
\newtheorem{Lemma}[Thm]{Lemma}
\newtheorem{Cor}[Thm]{Corollary}

\DeclareMathOperator{\var}{\mathsf{var}}

\begin{document}

\title{Unary enhancements\\ of inherently non-finitely based semigroups}
\author{K.~Auinger}
\address{Fakult\"at f\"ur Mathematik, Universit\"at Wien, Nordbergstrasse 15,  A-1090 Wien, Austria}
\email{karl.auinger@univie.ac.at}

\author{I.~Dolinka}
\address{ Department of Mathematics and Informatics, University of Novi Sad, Trg Do\-si\-teja Obradovi\'ca 4, 21101 Novi Sad, Serbia}
\email{dockie@dmi.uns.ac.rs}

\author{T.~V.~Pervukhina}
\address{Institute of Mathematics and Computer Science, Ural Federal University, Lenina 51, 620000 Ekaterinburg, Russia}
\email{cristofory@gmail.com, mikhail.volkov@usu.ru}

\author{M.~V.~Volkov}

\maketitle

\begin{abstract}
This paper is a continuation of ~\cite{ADV:2012}, more precisely, of Subsection 2.2 of~\cite{ADV:2012} dealing with \infb\ \iss. We exhibit
a simple condition under which a finite \is\ whose \sgp\ reduct is \infb\ is also \infb\ as a unary \sgp. As applications, we get already
known as well as new examples of \infb\ \iss. We also show that for finite regular semigroups, our condition is not only sufficient but
also necessary for the property of being \infb\ to persist. This leads to an algorithmic description of regular \infb\ \iss.
\end{abstract}

\section{Background and overview}

The \fbp, that is, the problem of classifying \sgps\ according to the finite basability of their identities, has been intensively explored
since the 1960s. Since the 1970s, the same problem has become investigated for \sgps\ endowed with an additional unary operation $x\mapsto
x^*$; such structures are commonly called \emph{unary \sgps}.

If $\mathcal{S}=\langle S,\cdot,{}^*\rangle$ is a unary \sgp, then the (plain) \sgp\ $\langle S,\cdot\rangle$ is called the (\emph{\sgp})
\emph{reduct} of $\mathcal{S}$. It is quite natural to ask how the answer to the \fbp\ for a given unary \sgp\ relates to the finite
basability of the identities of its reduct. The question turns out to be somewhat delicate. On the one hand, when we enhance the vocabulary
of an equational language by adding a unary operation, the expressive power of the language increases. Hence $\mathcal{S}$ usually has more
identities than $\langle S,\cdot\rangle$ so that the former may have more chance to become \nfb. On the other hand, the inference power of
the language increases too. Hence one can imagine the situation when some identity of $\langle S,\cdot\rangle$  does not follow from an
identity system $\Sigma$ as a plain identity but follows from $\Sigma$ when treated as a unary identity. This indicates that $\mathcal{S}$
may be \fb\ even if $\langle S,\cdot\rangle$ is not. The cumulative effect of the trade-off between increased expressivity and increased
inference power is hard to predict in general, and both possible outcomes indeed occur. This means that there exist unary \sgps, even
groups $\mathcal{G}=\langle G;\cdot,{}^{-1}\rangle$ with inversion as the unary operation, such that $\mathcal{G}$ is \fb\ [\nfb] as a
group while its reduct $\langle G;\cdot\rangle$ is \nfb\ [respectively, \fb] as a plain \sgp. See \cite[Section~2]{Volkov:2001} for
concrete examples (known since the 1970s), references and a discussion.

Much attention has been paid to the restriction of the \fbp\ to the class of \fss, in the plain as well as the unary setting, see, e.g.,
the survey \cite{Volkov:2001}. Therefore it appears a bit surprising that the above question about the relation between the finite
basability of a unary \sgp\ and of its reduct has not been systematically explored in the realm of \fss. To the best of our knowledge, the
first example of a \nfb\ \fus\ whose reduct is \fb\ was constructed only in~1998, see~\cite{Lawrence&Willard:1998}. The unary operation
used in~\cite{Lawrence&Willard:1998} was rather ad hoc, and similar examples with well behaved unary operations (including an example of a
\nfb\ \fis\ with \fb\ reduct) have only recently appeared in \cite{Jackson&Volkov:2010}. Examples of the `opposite' kind (of \fb\ \fuss\
with \nfb\ reducts) are not yet known.

For \fss, the following strengthening of the property of being \nfb\ has been successfully studied. Recall that a variety $\mathbf{V}$ of
[unary] \sgps\ is called \emph{locally finite} if every \fg\ [unary] \sgp\ in $\mathbf{V}$ is finite. A finite [unary] \sgp\ is said to be
\emph{\infb} (INFB for short) if it is not contained in any locally finite \fb\ variety of [unary] \sgps. Since the variety generated by a
finite [unary] \sgp\ is known to be locally finite, an INFB [unary] \sgp\ certainly is \nfb. In fact, the property of being INFB is much
stronger than the property of being \nfb\ and also behaves more regularly, see \cite{Volkov:2001} for details.

Sapir~\cite{Sapir:1987} has given an efficient (in the algorithmic sense of the word) description of INFB \sgps. INFB unary \sgps\ have
been investigated in~\cite{Dolinka:2010,ADV:2012} where some sufficient and some necessary conditions for a \fis\ to be INFB have been
found. Again, in this situation it is quite natural to ask what happens when one passes from a \fus\ to its reduct. The aforementioned
example of~\cite{Lawrence&Willard:1998} is in fact INFB so that in general an INFB unary \sgp\ may have a \fb\ reduct. This is however
impossible for a \fis; indeed, it is easy to verify (see Lemma~\ref{easy} below) that the reduct of an INFB \is\ must be INFB. The converse
is not true as first observed in~\cite{Sapir:1993}, and it is this circumstance that gives rise to the specific question addressed in the
present paper: when does an involution $x\mapsto x^*$ defined on an INFB \sgp\ $\langle S,\cdot\rangle$ preserve the property of being INFB
in the sense that the resulting \is\ $\mathcal{S}=\langle S,\cdot,{}^*\rangle$ is INFB as a unary \sgp?  We show (Theorem~\ref{twisted})
that this is the case whenever the variety generated by $\mathcal{S}$ contains a certain 3-element \is\ $\TSL$ (twisted semilattice). This
result has several applications: first, we give new, easy and uniform proofs for some examples of INFB \iss\ found in~\cite{ADV:2012};
second, we exhibit a further series of INFB \iss. We also show (Corollary~\ref{characterization}) that if $\langle S,\cdot\rangle$ is a finite regular
\sgp, then the presence of the 3-element twisted semilattice $\TSL$ in the variety generated by $\mathcal{S}$ is not only sufficient but
also necessary for the property of being INFB to persist. Combined with Sapir's result, this leads to an efficient description of regular
INFB \iss.

\section{Preliminaries}
\label{sec:preliminaries}

We assume the reader's acquaintance with basic concepts of universal algebra such as the notion of a variety and the HSP-theorem, see,
e.g., \cite[Chapter~II]{BuSa81}. Section~\ref{sec:regular} also requires some knowledge of Green's relations, cf.~\cite[Chapter~2]{how}.

A unary \sgp\ $\mathcal{S}=\langle S,\cdot,{}^*\rangle$ is called an \emph{\is} if it satisfies the identities
\begin{equation}
\label{eq:invlution} (xy)^*=y^*x^* \quad \text{ and  }\quad (x^*)^*=x,
\end{equation}
in other words, if the unary operation $x\mapsto x^*$ is an involutory anti-automorphism of the reduct $\langle S,\cdot\rangle$.

The \emph{free \is} $\FI(X)$ on a given alphabet $X$ can be constructed as follows.  Let $\overline{X}=\{x^*\mid x\in X\}$ be a disjoint
copy of $X$. We refer to the elements of $X$ as \emph{plain letters} and to the elements of $\overline{X}$ as \emph{starred letters}.
Define $(x^*)^*=x$ for all $x^*\in \overline{X}$. Then $\FI(X)$ is the free \sgp\ $(X\cup\overline{X})^+$ endowed with the involution
defined by
$$(x_1\cdots x_m)^* = x_m^*\cdots x_1^*$$
for all $x_1,\dots,x_m\in X\cup \overline{X}$. We refer to elements of $\FI(X)$ as \emph{involutory words over $X$} while elements of the
free semigroup $X^+$ will be referred to as \emph{plain words over $X$}.

If an \is\ $\mathcal{T}=\langle T,\cdot,{}^*\rangle$ is generated by a set $Y\subseteq T$, then every element in $\mathcal{T}$ can be
represented by an involutory word over $Y$ and thus by a plain word over $Y\cup\overline{Y}$ where $\overline{Y}=\{y^*\mid y\in Y\}$. Hence
the reduct $\langle T,\cdot\rangle$ is generated by the set $Y\cup\overline{Y}$; in particular, $\mathcal{T}$ is \fg\ if and only if so is
$\langle T,\cdot\rangle$. This observation immediately leads to the following fact already mentioned in the introduction.

\begin{Lemma}
\label{easy} If an \is\ $\mathcal{S}=\langle S,\cdot,{}^*\rangle$ is \infb, then so is its reduct $\langle S,\cdot\rangle$.
\end{Lemma}

\begin{proof}
Arguing by contradiction, assume that $\langle S,\cdot\rangle$ is not INFB. Then $\langle S,\cdot\rangle$ belongs to a locally finite plain
\sgp\ variety defined by a finite identity system $\Sigma$. Consider the variety $\mathbf{V}$ of \iss\ defined by the identities
\eqref{eq:invlution} and $\Sigma$. Clearly, $\mathbf{V}$ is \fb\ and $\mathcal{S}\in\mathbf{V}$. If $\mathcal{T}=\langle
T,\cdot,{}^*\rangle$ is a \fg\ \is\ from $\mathbf{V}$ then the reduct $\langle T,\cdot\rangle$ is a \fg\ plain \sgp\ by the observation
preceding the formulation of the lemma. Since the reduct satisfies the identities in $\Sigma$ and $\Sigma$ defines a locally finite plain
\sgp\ variety, we conclude that the base set $T$ is finite. Hence the variety $\mathbf{V}$ is locally finite and $\mathcal{S}$ belongs to a
locally finite \fb\ variety, a contradiction.
\end{proof}

As mentioned, the converse of Lemma~\ref{easy} is not true in general. For an example, consider the well known \emph{Brandt monoid}
$\langle B_2^1,\cdot\rangle$, where $B_2^1$ is the set of the following six integer $2\times 2$-matrices:
$$\begin{pmatrix} 0 & 0\\ 0 & 0\end{pmatrix},\
\begin{pmatrix} 1 & 0\\ 0 & 0\end{pmatrix},\
\begin{pmatrix} 0 & 1\\ 0 & 0\end{pmatrix},\
\begin{pmatrix} 0 & 0\\ 1 & 0\end{pmatrix},\
\begin{pmatrix} 0 & 0\\ 0 & 1\end{pmatrix},\
\begin{pmatrix} 1 & 0\\ 0 & 1\end{pmatrix},$$
and the binary operation $(A_1,A_2)\mapsto A_1\cdot A_2$ is the usual matrix multiplication. It is known
\cite[Corollary~6.1]{sapirburnside} that the Brandt monoid is INFB (this was in fact the very first example of an INFB \sgp). The Brandt
monoid admits a natural involution, namely, the usual matrix transposition $A\mapsto A^T$. However, the \is\ $\langle
B_2^1,\cdot,{}^T\rangle$ is not INFB as shown in~\cite{Sapir:1993}. Further examples can be found in~\cite{ADV:2012}: if $\mathcal{K}$ is a
finite field and $\mathrm{M}_n(\mathcal{K})$ stands for the set of all $n\times n$-matrices over $\mathcal{K}$, then the \sgp\
$\langle\mathrm{M}_n(\mathcal{K}),\cdot\rangle$ is INFB for any $n\ge 2$ by \cite[Corollary~6.2]{sapirburnside} while the \is\
$\langle\mathrm{M}_2(\mathcal{K}),\cdot,{}^T\rangle$ is not INFB if the number of elements in $\mathcal{K}$ leaves reminder~3 when divided
by~4.

Thus, not every involution defined on an INFB \sgp\ preserves the property of being INFB, and we are looking towards a classification of
`INFB-preserving' involutions. Our main tool is the following result from~\cite{ADV:2012}. Recall the notions that appear in its
formulation. Let $x_1,x_2,\dots,x_n,\dots$ be a sequence of letters. The sequence $\{Z_n\}_{n=1,2,\dots}$ of \emph{Zimin words} is defined
inductively by $Z_1=x_1$, $Z_{n+1}=Z_nx_{n+1}Z_n$. We say that an involutory word $v$ is an \emph{involutory isoterm for a unary semigroup
$\mathcal{S}$} if the only involutory word $v'$ such that $\mathcal{S}$ satisfies the \is\ identity $v=v'$ is the word $v$ itself.

\begin{Thm}[{\mdseries\cite[Theorem~2.3]{ADV:2012}}]
\label{Theorem 2.2} Let $\mathcal{S}$ be a \fis. If all Zimin words are involutory isoterms for $\mathcal{S}$, then $\mathcal{S}$ is \infb.
\end{Thm}

\section{Main result and its applications}
\label{sec:main}

Recall that \sgps\ satisfying both $xy=yx$ and $x^2=x$ are called \emph{semilattices}. An \is\ $\mathcal{S}=\langle S,\cdot,{}^*\rangle$
whose reduct $\langle S,\cdot\rangle$ is a semilattice with 0 is said to be a \emph{twisted semilattice} if 0 is the only fixed point of
the involution $x\mapsto x^*$. This class of \iss\ was first considered in~\cite{Fajt}. It is easy to see that the minimum non-trivial
object in this class is the 3-element twisted semilattice $\TSL=\langle\{e,f,0\},\cdot,{}^*\rangle$ in which $e^2=e$, $f^2=f$ and all other
products are equal to $0$, while the unary operation is defined by $e^*=f$, $f^*=e$, and $0^*=0$.

If $ \mathcal{S}$ is an \is, we denote by $\var\mathcal{S}$ the variety generated by $\mathcal{S}$.

\begin{Thm}\label{sufficientINFB}
\label{twisted} Let $\mathcal{S}=\langle S,\cdot,{}^*\rangle$ be a \fis\ such that $\TSL\in\var\mathcal{S}$. If the reduct $\langle
S,\cdot\rangle$ is \infb, then so is $\mathcal{S}$.
\end{Thm}

\begin{proof}
By Theorem~\ref{Theorem 2.2} we only have to show that $\mathcal{S}$ satisfies no non-trivial \is\ identity of the form $Z_n=z$. If $z$ is
a plain word, we can refer to~\cite[Proposition~7]{sapirburnside} according to which the INFB \sgp\ $\langle S,\cdot\rangle$ satisfies no
non-trivial plain \sgp\ identity of the form $Z_n=z$. Now suppose that $\mathcal{S}$ satisfies an identity $Z_n=z$ such that the involutory
word $z$ is not a plain word. This means that $z$ contains a starred letter. Since $\TSL\in\var\mathcal{S}$, the identity $Z_n=z$ holds in
$\TSL$. Substitute the element $e$ of $\TSL$ for all plain letters occurring in $Z_n$ and $z$. Since $e^2=e$, the value of the word $Z_n$
under this substitution equals $e$. On the other hand, since $z$ contains a starred letter, the value of $z$ is a product involving
$e^*=f$, and every such product is equal to either $f$ or 0. This is a contradiction.
\end{proof}

As for applications of Theorem~\ref{twisted}, we first give simplified and uniform proofs for two important results from \cite{ADV:2012}.
To start with, consider the \emph{twisted Brandt monoid} $\TB=\langle B_2^1,\cdot,{}^D\rangle$ arising when one endows the Brandt monoid
$\langle B_2^1,\cdot\rangle$ with the unary operation $A\mapsto A^D$ that fixes the matrices $\left(\begin{smallmatrix} 0 & 0\\ 0 &
0\end{smallmatrix}\right),\ \left(\begin{smallmatrix} 0 & 1\\ 0 & 0\end{smallmatrix}\right),\ \left(\begin{smallmatrix} 0 & 0\\ 1 &
0\end{smallmatrix}\right),\ \left(\begin{smallmatrix} 1 & 0\\ 0 & 1\end{smallmatrix}\right)$ and swaps each of the matrices
$\left(\begin{smallmatrix}1 & 0\\ 0 & 0\end{smallmatrix}\right),\ \left(\begin{smallmatrix} 0 & 0\\ 0 & 1\end{smallmatrix}\right)$ with the
other one. We notice that this unary operation is just the reflection with respect to the secondary diagonal (from the top right to the
bottom left corner). The reflection (called the \emph{skew transposition}) makes sense for every square matrix and is in fact an involution
of the \sgp\ $\langle\mathrm{M}_n(\mathcal{K}),\cdot\rangle$; this follows from the observation that for every matrix
$A\in\mathrm{M}_n(\mathcal{K})$, one has $A^D=JA^TJ$, where $J$ is the $n\times n$-matrix with 1s on the secondary diagonal and 0s
elsewhere. Moreover suppose that the field $\mathcal{K}$ is such that there exists a matrix $R\in\mathrm{M}_n(\mathcal{K})$ satisfying
$R^T=R$ and $R^2=J$ (this happens, e.g., when the characteristic of $\mathcal{K}$ is not  $2$ and  square roots of ${-1}$ and ${2}$ do
exist in $\mathcal{K}$). Then the conjugation map $A\mapsto A\psi:= R^{-1}AR$ satisfies $(A^D)\psi=(A\psi)^T$ and hence is an isomorphism
between the \iss\ $\langle\mathrm{M}_n(\mathcal{K}),\cdot,{}^D\rangle$ and $\langle\mathrm{M}_n(\mathcal{K}),\cdot,{}^T\rangle$. Clearly,
the set $B_2^1$ can be considered as a subset of $\mathrm{M}_2(\mathcal{K})$, and as such it is closed under both the usual transposition
and the skew one. Therefore it appears a bit surprising that the involutory sub\sgps\ $\TB=\langle B_2^1,\cdot,{}^D\rangle$ and $\langle
B_2^1,\cdot,{}^T\rangle$ of the (isomorphic) \iss\ $\langle\mathrm{M}_2(\mathcal{K}),\cdot,{}^D\rangle$ and respectively
$\langle\mathrm{M}_2(\mathcal{K}),\cdot,{}^T\rangle$ turn out to be so much different. Indeed, $\langle B_2^1,\cdot,{}^T\rangle$ is not
INFB (see Section~\ref{sec:preliminaries}) while $\TB$ is, as the following corollary reveals.

\begin{Cor}[{\mdseries\cite[Corollary 2.7]{ADV:2012}}]
\label{twisted Brandt} The twisted Brandt monoid $\TB$ is inherently \nfb.
\end{Cor}

\begin{proof}
As already mentioned, the reduct $\langle B_2^1,\cdot\rangle$ of $\TB$ is INFB by~\cite[Corollary~6.1]{sapirburnside}. The matrices
$\left(\begin{smallmatrix}1 & 0\\ 0 & 0\end{smallmatrix}\right)$, $\left(\begin{smallmatrix}0 & 0\\ 0 & 1\end{smallmatrix}\right)$, and
$\left(\begin{smallmatrix}0 & 0\\ 0 & 0\end{smallmatrix}\right)$ form an involutory sub\sgp\ in $\TB$ and, obviously, this sub\sgp\ is
isomorphic to the 3-element twisted semilattice $\TSL$. Thus, Theorem~\ref{twisted} applies.
\end{proof}

Now consider the matrix \iss\ $\langle\mathrm{M}_n(\mathcal{K}),\cdot,{}^T\rangle$ where $\mathcal{K}$ is a finite field.

\begin{Cor}[{\mdseries\cite[Theorems 3.9 and 3.10]{ADV:2012}}]
\label{matrices} The \is\ $\langle\mathrm{M}_n(\mathcal{K}),\cdot,{}^T\rangle$, where $\mathcal{K}$ is a finite field, is \infb\ if $n\ge
3$ or if $n=2$ and the number of elements in $\mathcal{K}$ is not of the form $4k+3$.
\end{Cor}

\begin{proof}
The reduct $\langle\mathrm{M}_n(\mathcal{K}),\cdot\rangle$ is INFB for each $n\ge2$ and each finite field $\mathcal{K}$
by~\cite[Corollary~6.2]{sapirburnside}. To invoke Theorem~\ref{twisted}, it only remains to show that, under the condition of the
corollary, the 3-element twisted semilattice $\TSL$ belongs to the variety $\var\langle\mathrm{M}_n(\mathcal{K}),\cdot,{}^T\rangle$.

First let $n\ge3$. By the Chevalley--Warning theorem \cite[Corollary~2 in \S1.2]{Serre}, the field $\mathcal{K}$ contains some elements
$x,y$ satisfying $1+x^2+y^2=0$. Then the $n\times n$-matrices
$$e=\begin{pmatrix}
1 & 0 & 0 & \cdots & 0\\
x & 0 & 0 & \cdots & 0\\
y & 0 & 0 & \cdots & 0\\
\vdots & \vdots & \vdots &\ddots & \vdots\\
0 & 0 & 0 & \cdots & 0
\end{pmatrix},\ \ f=
\begin{pmatrix} 1 & x & y & \cdots & 0\\
0 & 0 & 0 & \cdots & 0\\
0 & 0 & 0 & \cdots & 0\\
\vdots & \vdots & \vdots &\ddots & \vdots\\
0 & 0 & 0 & \cdots & 0
\end{pmatrix},\
g=
\begin{pmatrix} 1 & x & y & \cdots & 0\\
x & x^2 & xy & \cdots & 0\\
y & xy & y^2 & \cdots & 0\\
\vdots & \vdots & \vdots &\ddots & \vdots\\
0 & 0 & 0 & \cdots & 0
\end{pmatrix}$$
satisfy
$$e^2=e,\ f^2=f,\ ef=g,\ fe=0,\ e^T=f,\ f^T=e,\ \text{ and } g^T=g.$$
Therefore the set $\{e,f,g,0\}$ forms an involutory sub\sgp\ in $\langle\mathrm{M}_n(\mathcal{K}),\cdot,{}^T\rangle$, the set $\{g,0\}$ is
an ideal of this sub\sgp\ and is closed under transposition. It remains to observe that the Rees quotient of the \is\
$\langle\{e,f,g,0\},\cdot,{}^T\rangle$ over the ideal $\{g,0\}$ is isomorphic to the 3-element twisted semilattice $\TSL$.

Now let $n=2$ and let the number of elements in $\mathcal{K}$ be not of the form $4k+3$. Then the field $\mathcal{K}$ contains a square
root of ${-1}$, see, e.g., \cite[Theorem~3.75]{LidlNiederreiter}. Now the argument of the previous paragraph applies, with the $2\times
2$-matrices
$$e'=\begin{pmatrix}
1 & 0 \\
x & 0
\end{pmatrix},\ \ f'=
\begin{pmatrix}
1 & x\\
0 & 0
\end{pmatrix},\ \text{ and }\
g'=
\begin{pmatrix}
1 & x \\
x & -1
\end{pmatrix}$$
in the roles of $e,f$, and $g$, respectively, where $x$ denotes some fixed square root of $-1$.
\end{proof}

We could have continued along the same lines to show that in fact all examples of INFB \iss\ found in~\cite{Dolinka:2010,ADV:2012} can be
similarly deduced from Theorem~\ref{twisted}. However we think that Corollaries~\ref{twisted Brandt} and~\ref{matrices} are representative
enough. Now we present a new application.

Let $\mathcal{K}$ be a finite field and let $\mathrm{T}_n(\mathcal{K})$ stand for the set of all upper-triangular $n\times n$-matrices over
$\mathcal{K}$. The set $\mathrm{T}_n(\mathcal{K})$ forms an \is\ under the usual matrix multiplication and the skew transposition. The
following result classifies all INFB \iss\ of the form $\langle\mathrm{T}_n(\mathcal{K}),\cdot,{}^D\rangle$.

\begin{Thm}
\label{triangular} The \is\ $\langle\mathrm{T}_n(\mathcal{K}),\cdot,{}^D\rangle$, where $\mathcal{K}$ is a finite field, is \infb\ if and
only if $n\ge 4$ and $\mathcal{K}$ contains at least $3$ elements.
\end{Thm}

\begin{proof}
In~\cite{Goldberg&Volkov:2003} it is shown that the reduct $\langle\mathrm{T}_n(\mathcal{K}),\cdot\rangle$ is INFB if and only if $n\ge 4$
and $\mathcal{K}$ contains at least $3$ elements. Therefore, the `only if' part of our theorem follows from Lemma~\ref{easy} and the `if'
part will follow from Theorem~\ref{twisted} as soon as we shall verify that
$\TSL\in\var\langle\mathrm{T}_n(\mathcal{K}),\cdot,{}^D\rangle$. Indeed, for every $n\ge2$ the matrix units $e_{11}$ and $e_{nn}$ belong to
$\mathrm{T}_n(\mathcal{K})$ and satisfy
$$e_{11}^2=e_{11},\ e_{nn}^2=e_{nn},\ e_{11}e_{nn}=e_{nn}e_{11}=0,\ e_{11}^D=e_{nn},\ \text{ and } e_{nn}^D=e_{11}.$$
Hence the set $\{e_{11},e_{nn},0\}$ forms an involutory sub\sgp\ in $\langle\mathrm{T}_n(\mathcal{K}),\cdot,{}^D\rangle$ and this
involutory sub\sgp\ is isomorphic to the 3-element twisted semilattice $\TSL$.
\end{proof}

Observe that in~\cite{Goldberg&Volkov:2003} it is shown that for any $n$ and $\mathcal{K}$, the Brandt monoid $\langle B_2^1,\cdot\rangle$
does not belong to the \sgp\ variety generated by the \sgp\ $\langle\mathrm{T}_n(\mathcal{K}),\cdot\rangle$. Hence the twisted Brandt
monoid $\TB$ does not belong to the \is\ variety $\var\langle\mathrm{T}_n(\mathcal{K}),\cdot,{}^D\rangle$. Thus, Theorem~\ref{triangular}
provides a series of examples of INFB \iss\ whose varieties do not contain $\TB$. Such examples have not been known before.

\section{Regular semigroups}
\label{sec:regular}

In this section we show that the presence of the 3-element twisted semilattice $\TSL$ in the variety generated by a \fis\ $\mathcal{S}$ is
(not only sufficient but also) necessary for $\mathcal{S}$ in order to inherit the property of being INFB  from its \sgp\ reduct, provided
that $\mathcal{S}$ is regular. As a preliminary result we present a criterion of whether or not $\TSL$ belongs to $\var\mathcal{S}$
(Corollary~\ref{TSLinvarS}).

We shall use two classical results concerning Green's relations, the first of which is often referred
to as the Lemma of Miller and Clifford (see \cite[Proposition~2.3.7]{how}), while the property formulated
in the second one is usually called the \emph{stability} of Green's relations (see \cite[Proposition~3.1.4 (2)]{pin}).

\begin{Lemma} \label{Miller&Clifford}
\begin{enumerate}
\item Let $a,b$ be elements of a $\mathscr{D}$-class of an arbitrary \sgp. Then $ab\in R_a\cap L_b$ if and only if
$L_a\cap R_b$ contains an idempotent.
\item Let $S$ be a finite semigroup and $a,b\in S$. Then $a\Jc ab$ implies $a\Rc ab$ and $b\Jc ab$ implies $b \Lc ab$.
\end{enumerate}
\end{Lemma}

The above mentioned criterion for membership of $\TSL$ in $\var\mathcal{S}$ is clarified by the following key result.
\begin{Prop}\label{alternative} For a finite involutory semigroup $\mathcal{S}$ exactly one of the two following assertions is true.
\begin{enumerate}
\item[(A)] There exists an idempotent $e$ of $\mathcal{S}$ satisfying $e\mathrel{{>}_{\!\!\!\Jc}} e^*e$.
\item[(B)] There exists a positive integer $N$ such that  $\mathcal{S}$ satisfies the identity
\begin{equation}\label{identityfor(B)}
x^N=(x^N(x^N)^*)^Nx^N.
\end{equation}
\end{enumerate}
\end{Prop}

\begin{proof} It is clear that the conditions (A) and (B) exclude each other. Let us assume that the assertion (A) does not hold
for $\mathcal{S}=\langle S,\cdot,{}^*\rangle$. We have to prove that $\mathcal{S}$ satisfies (B). For each idempotent $e$ of $\mathcal{S}$
we have $e\Jc e^*e$ and therefore $e\Lc e^*e$ by Lemma~\ref{Miller&Clifford}~(2). Since the involution ${}^*$ is an anti-automorphism, we
also have $e^*\Rc e^*e$ for each idempotent $e$. Swapping the roles of $e$ and $e^*$ we also get $e\Rc ee^*\Lc e^*$. In other words,
$$e\Rc ee^*\Lc e^*\Rc e^*e\Lc e$$ holds for each idempotent $e$ of $\mathcal{S}$.

By the `only if' part of Lemma~\ref{Miller&Clifford} (1), the fact that the product $e^*e$ belongs to $L_e\cap R_{e^*}$ implies that the
$\mathscr{H}$-class $H_{ee^*}=R_e\cap L_{e^*}$ contains an idempotent $g$ and hence, by Green's theorem~\cite[Theorem~2.2.5]{how}, this
class is a subgroup of $\langle S,\cdot\rangle$ having $g$ as its identity element. Since $g\Rc e$ we have that $ge=e$. Now take any common
multiple $n$ of the exponents of all subgroups of $\mathcal{S}$; then $(ee^*)^n=g$ and hence $(ee^*)^ne=e$.  Finally, choose a positive
integer $N$ for which $\mathcal{S}$ satisfies the identity $x^N=x^{2N}$. Then $N$ is a common multiple of the exponents of all subgroups of
$\mathcal{S}$ and each element of the form $x^N$ is idempotent. Consequently $\mathcal{S}$ satisfies the identity \eqref{identityfor(B)}.
\end{proof}

It is now easy to see that $\TSL\in\var\mathcal{S}$ if and only if $\mathcal{S}$ is of type (A). Indeed suppose
that $\mathcal{S}$ has an idempotent $e$ satisfying $e\mathrel{{>}_{\!\!\!\Jc}} e^*e$ and let $\mathcal{T}=\langle T,\cdot,{}^*\rangle$
be the involutory subsemigroup of $\mathcal{S}$ generated by $e$; then $e\ne e^*$ and neither of the idempotents $e$ and $e^*$
is contained in the ideal $I:=Tee^*T\cup Te^*eT$. It follows that $\TSL$ is isomorphic to the Rees quotient $\mathcal{T}/I$.
In other words, $\TSL$ is a homomorphic image of an involutory subsemigroup of $\mathcal{S}$, that is, $\TSL$ \emph{divides}
$\mathcal{S}$ and in particular $\TSL\in \var\mathcal{S}$.

Conversely, if $\mathcal{S}$ is of type (B) then  $\mathcal{S}$ satisfies the identity (\ref{identityfor(B)})
 for some positive integer $N$. Obviously, $\TSL$ does not satisfy this identity and, hence, it does not belong to $\var\mathcal{S}$.

Altogether we have proved:
\begin{Cor}\label{TSLinvarS} For a finite involutory semigroup $\mathcal{S}$, the following are equivalent:
\begin{enumerate}
\item $\mathcal{S}$ is of type (A).
\item $\TSL$ divides $\mathcal{S}$.
\item $\TSL\in \var\mathcal{S}$.
\end{enumerate}
\end{Cor}

A finite involutory semigroup $\mathcal{S}$ of type (A) with INFB semigroup reduct $\langle S,\cdot\rangle$ is INFB as
an involutory semigroup (Theorem \ref{sufficientINFB}). At the time of writing, the authors were not (yet) aware of an
example of an INFB involutory semigroup of type (B).

We note that finite involutory semigroups of type (B) can be characterized in various ways; for example, as those in which each regular
element admits a Moore--Penrose inverse, and likewise, as those in which each regular ${\Lc}$-class (and/or each regular ${\Rc}$-class)
contains a \emph{projection} (that is, an idempotent fixed under the involution) \cite{NP}. Another equivalent condition is that each
involutory subsemigroup $\langle g\rangle$ generated by a single idempotent $g$ is completely simple. Moreover, the class of all finite
involutory semigroups of type (B) forms a pseudovariety of involutory semigroups, namely the one defined by the pseudoidentity
$$x^\omega=(x^\omega(x^\omega)^*)^\omega x^\omega.$$
As usual, $s\mapsto s^\omega$ denotes the unary operation that assigns to each element $s$ of a finite semigroup the unique idempotent in
the cyclic subsemigroup generated by $s$.

Recall that an element $x$ of a \sgp\ $\langle S,\cdot\rangle$ is said to be \emph{regular} if $x=xyx$ for some $y\in S$. A [unary] \sgp\
is \emph{regular} if all of its elements are regular. We shall refine the proof of Proposition \ref{alternative} and show that a
\textbf{regular} involutory semigroup $\mathcal{S}$ of type (B) satisfies an identity that guarantees that $\mathcal{S}$ is \textbf{not}
INFB, thanks to the following result from~\cite{ADV:2012}.

\begin{Prop}[{\mdseries\cite[Proposition~2.9]{ADV:2012}}]
\label{NINFB} Let $\mathcal{S}=\langle S,\cdot,{}^*\rangle$ be a \fis. If there exists an involutory word $\iota\!(x)$ in one variable $x$ such
that $\mathcal{S}$ satisfies the identity $x=x\iota\!(x)x$, then $\mathcal{S}$ is not \infb.
\end{Prop}
We get the following consequence:
\begin{Cor}\label{regular}
A finite regular involutory semigroup $\mathcal{S}$ of type (B) is not inherently nonfinitely based.
\end{Cor}
\begin{proof}
We are going to sharpen the proof of Proposition~\ref{alternative}. Let $\mathcal{S}$ be an  involutory semigroup of type (B) (not
necessarily regular at this point). Fix an arbitrary regular element $x\in S$  and take an element $y\in S$   such that $x=xyx$. Then
$e=xy$ and $f=yx$ are idempotents and $e\Rc x\Lc f$. Since the involution is an anti-automorphism of $\langle S,\cdot\rangle$, we also have
$e^*\Lc x^*\Rc f^*$. We have already seen in the proof of Proposition \ref{alternative} that $ee^*\Rc e\Lc e^*e\Rc e^*\Lc ee^*$ and
$ff^*\Rc f\Lc f^*f\Rc f^*\Lc ff^*$.  All listed relations are graphically represented in Fig.~\ref{fig:D-class} that shows an appropriate
fragment of the eggbox picture for the $\mathscr{D}$-class of $x$ and $x^*$.
\begin{figure}[th]
\begin{center}
{\large \begin{tabular}{|c|c|c|c|}

\hline $f^*f$\rule[-5pt]{0pt}{16pt} & & $f^*$ & $x^*$ \\
\hline \rule[-5pt]{0pt}{16pt}& $e^*e$  & & $e^*$\\
\hline $f$\rule[-5pt]{0pt}{16pt} & \phantom{$f^*f$} & $ff^*$ &  \phantom{$f^*f$}\\
\hline $x$\rule[-5pt]{0pt}{16pt} & $e$ & & $ee^*$\\
\hline
\end{tabular}}
\caption{A fragment of the eggbox picture for the $\mathscr{D}$-class of the elements $x$ and $x^*$}\label{fig:D-class}
\end{center}
\end{figure}

As in the proof of Proposition~\ref{alternative}, the $\mathscr{H}$-class $H_{ee^*}=R_e\cap L_{e^*}$ contains an idempotent $g$ and again
this class is a subgroup of $\langle S,\cdot\rangle$ having $g$ as its identity element. Observe that $R_e=R_x$ and $L_{e^*}=L_{x^*}$
whence $H_{ee^*}=R_x\cap L_{x^*}$. Also observe that $gx=x$ since $g$ is an idempotent and $g\Rc x$.

Similarly, $ff^*\in R_f\cap L_{f^*}$ implies that $H_{f^*f}=R_{f*}\cap L_f$ contains an idempotent. However, $R_{f^*}=R_{x^*}$ and
$L_f=L_x$. Now, by the `if' part of Lemma~\ref{Miller&Clifford} (1), the fact that $L_x\cap R_{x^*}=H_{f^*f}$ contains an idempotent implies
that the product $xx^*$ lies in $R_x\cap L_{x^*}=H_{ee^*}$. Let $n$ be the least common multiple of the exponents of the subgroups of $\langle
S,\cdot\rangle$. We then have $(xx^*)^n=g$ whence $x=gx=(xx^*)^nx$.

Consequently, if $\mathcal{S}$ is regular then $\mathcal{S}$ satisfies the identity $x=(xx^*)^nx$ and hence
$x=x\iota\!(x)x$ for $\iota\!(x)=x^*(xx^*)^{n-1}$, which by Proposition \ref{NINFB} implies that $\mathcal{S}$ is not INFB.
\end{proof}

We can now easily deduce various characterizations of regular INFB \iss.
\renewcommand{\labelenumi}{\emph{(\roman{enumi})}}
\begin{Cor}
\label{characterization} Let $\mathcal{S}=\langle S,\cdot,{}^*\rangle$ be a finite regular \is. Then the following are equivalent:
\begin{enumerate}
\item $\mathcal{S}$ is \infb;
\item the reduct $\langle S,\cdot\rangle$ is \infb\ and the $3$-element twisted semilattice $\TSL$ belongs to $\var\mathcal{S}$;
\item the reduct $\langle S,\cdot\rangle$ is \infb\ and there exists an idempotent $e$ satisfying $e\mathrel{{>}_{\!\!\!\Jc}} e^*e$;
\item all Zimin words are involutory isoterms for $\mathcal{S}$.
\end{enumerate}
\end{Cor}

\begin{proof}
(i) $\to$ (ii) follows from Lemma~\ref{easy}, Proposition \ref{alternative} and Corollary~\ref{regular}.

(ii) $\to$ (iv) has been verified in the course of the proof of Theorem~\ref{twisted}.

(iv) $\to$ (i) is Theorem~\ref{Theorem 2.2}.

(ii) $\leftrightarrow$ (iii) follows from Proposition \ref{alternative} and Corollary \ref{TSLinvarS}.
\end{proof}

We observe that the condition (iii) in Corollary~\ref{characterization} is algorithmically verifiable. Indeed, given a finite
regular \is\ $\mathcal{S}$, we can check whether or not its reduct is INFB by using Sapir's algorithm from~\cite{Sapir:1987},
and the condition on the idempotents is obviously decidable.

\begin{Cor}
There exists an algorithm which decides, when given a finite regular \is\ $\mathcal{S}$, whether or not $\mathcal{S}$ is \infb.
\end{Cor}

\paragraph*{\textbf{Acknowledgement}.} The authors are indebted to an anonymous referee for several inspiring remarks that have led to an improved
presentation of Section 4.

\end{document}